\documentclass[a4paper,12pt]{article}
\usepackage[longend,noline]{algorithm2e}
\usepackage{amsfonts}
\usepackage{amsmath}
\usepackage{amsthm}
\usepackage{enumerate}
\usepackage{todonotes}

\title{\textbf{Approximating the position of a hidden agent in a graph}}
\author{Hannah Guggiari\thanks{Mathematical Institute, University of Oxford, Oxford, OX2 6GG, United Kingdom. Email:\texttt{\{guggiari,robertsa,scott\}@maths.ox.ac.uk}}, Alexander Roberts\footnotemark[1], Alex Scott\footnotemark[1] \thanks{Supported by a Leverhulme Trust Research Fellowship.}}
\date{\today}

\newtheorem{theorem}{Theorem}[section]
\newtheorem{conjecture}[theorem]{Conjecture}
\newtheorem{lemma}[theorem]{Lemma}
\newtheorem{claim}[theorem]{Claim}
\newtheorem{prop}[theorem]{Proposition}

\newtheorem{question}[theorem]{Question}

\theoremstyle{definition}

\newcommand{\bN}{\mathbb{N}}
\newcommand{\cR}{\mathcal{R}}
\newcommand{\wt}{\widetilde}
\newcommand{\rad}{\textup{rad}}
\newcommand{\diam}{\textup{diam}}

\makeatletter
\def\blfootnote{\xdef\@thefnmark{}\@footnotetext}
\makeatother

\begin{document}
\newpage
\maketitle
\blfootnote{With thanks to the Bellairs Research Institute where some of this work was done.}
\begin{abstract}
\noindent
A cat and mouse play a pursuit and evasion game on a connected graph $G$ with $n$ vertices. The mouse moves to vertices $m_1,m_2,\dots$ of $G$ where $m_i$ is in the closed neighbourhood of $m_{i-1}$ for $i\geq2$. The cat tests vertices $c_1,c_2,\dots$ of $G$ without restriction and is told whether the distance between $c_i$ and $m_i$ is at most the distance between $c_{i-1}$ and $m_{i-1}$. The mouse knows the cat's strategy, but the cat does not know the mouse's strategy. We will show that the cat can determine the position of the mouse up to distance $O(\sqrt{n})$ within finite time and that this bound is tight up to a constant factor. This disproves a conjecture of Dayanikli and Rautenbach.
\end{abstract}
\noindent
\section{Introduction}
There are a wide variety of different pursuit and evasion games -- games where one player tries to localise a second player moving along the edges of a graph by using information about the current position of the second player.

One of the most widely studied games is Cops and Robbers, which was introduced by Nowakowski and Winkler \cite{NW83} and independently by Quilliot \cite{Q78}. One player controls a set of cops and the other player is the robber. At the start of the game, the cops are placed on vertices of $G$. The robber then chooses a vertex. The players take alternate turns with the cops playing first. On the cops' turn, each cop moves to a vertex within the closed neighbourhood of their current position and multiple cops may occupy the same vertex. The robber plays in a similar fashion, moving to a vertex within its current closed neighbourhood. The cops win if one of them is on the same vertex as the robber after a finite period of time; otherwise the robber wins. In this game, the players have perfect information and so they know the location of both the cops and robber at all times. One of the main open questions is Meyniel's Conjecture which states that, on any $n$-vertex graph, $O(\sqrt{n})$ cops can catch the robber. See \cite{BB12} for a general account and \cite{FKL12,LP12,SS11} for the best current bounds.

Haslegrave \cite{H14} proposed a different game, Cat and Mouse, in which the cat wins by finding the mouse in finite time. As in Cops and Robbers, the players take alternate turns. However, in this game, the cat has no information about where the mouse is and, on its turn, it is not constrained by the edges of $G$ and may test any vertex to see if the mouse is there. The mouse knows the cat's strategy in advance. The mouse is also required to move to a neighbouring vertex on each turn. These differences mean that the strategies employed by the players in Cat and Mouse differ greatly from those used in Cops and Robbers.

In this paper, we will study a variation of the Cat and Mouse pursuit and evasion game introduced by Dayanikli and Rautenbach \cite{DR18}. The rules are as described above, except the mouse is not required to move on its turn. A time step consists of one turn for each player with the mouse playing first. After the mouse has moved, the cat chooses any vertex and is told whether the distance between itself and the mouse has increased compared with the previous time step. The cat wins if it is able to localise the mouse up to a given distance $d$ within finite time (i.e. it can determine a vertex $v$ such that the mouse is at distance at most $d$ from $v$). Otherwise, the mouse wins. A more formal description of the game, including a precise definition of what it means for the cat to localise the mouse, is given in Section \ref{sec:game}.

Dayanikli and Rautenbach \cite{DR18} proved that, if $G$ is a tree with maximum degree $\Delta$ and radius $h$, then the cat can localise the mouse up to distance $4\Delta-6$ within time $O(h\Delta)$. They also showed that the cat can localise the mouse up to distance 8 within time $O(\log n)$ on the $n\times n$ grid. These results led them to make the following conjecture.
\begin{conjecture}
\label{conj:logn}
The cat can localise the mouse up to distance $O(\log n)$ on any connected graph of order $n$.
\end{conjecture}
\noindent
In Section \ref{sec:lower}, we will show that Conjecture \ref{conj:logn} is false. We will prove the following theorem.
\begin{theorem}
\label{thm:lower}
For all $n$ sufficiently large, there exists a tree $T$ on $n$ vertices such that, regardless of the strategy employed by the cat, the mouse is able to avoid being localised up to distance $\sqrt{n}/12$ on $T$ for an infinite period of time.
\end{theorem}
\noindent
In Section \ref{sec:upper}, we will prove that the cat can localise the mouse up to distance $O(\sqrt{n})$ on any simple, undirected and connected graph of order $n$ within time $O(\sqrt{n})$. In particular, we will prove the following theorem.
\begin{theorem}
\label{thm:upper}
The cat can localise the mouse up to distance $\sqrt{32n}$ by time $\sqrt{2n}$.
\end{theorem}

\noindent
Together, these two results show that the cat can localise the mouse up to distance $O(\sqrt{n})$ on any simple, undirected and connected graph of order $n$, and that this bound is tight up to a constant factor.

\section{Rules of the Game}
\label{sec:game}
There are two players, the \textit{cat} and the \textit{mouse}, and the game is played on a simple, undirected, connected graph $G$ of order $n$. It proceeds in discrete time steps, which are labelled by the positive integers.

The mouse is only able to move along the edges of $G$. Let $m_i$ be the position of the mouse at time $i$. For $i\geq2$, $m_i$ belongs to the closed neighbourhood of $m_{i-1}$, that is $m_i$ is either $m_{i-1}$ or a neighbour of $m_{i-1}$.

The cat may test any vertex of $G$ without any restrictions. At time $i$, the cat chooses a vertex $c_i$. At the end of time step $i\geq2$, the cat is told the value of $b_i$ where
\begin{equation*}
b_i=
\begin{cases}
1&\text{if }d(c_i,m_i)\leq d(c_{i-1},m_{i-1})\\
0&\text{if }d(c_i,m_i)>d(c_{i-1},m_{i-1}).
\end{cases}
\end{equation*}
where $d(u,v)$ denotes the length of a shortest path connecting vertices $u$ and $v$. The value of $b_i$ tells the cat whether it is further away from the mouse at step $i$ than it was at step $i-1$. The cat has no information about the location of the mouse when it chooses $c_1$ and $c_2$. However, for $i\geq3$, the cat can use $b_1,\dots,b_{i-1}$ to help it decide which vertex to choose for $c_i$. For each $i$, the cat can calculate $M_i$, the set of possible positions for $m_i$. A vertex $v$ belongs to $M_i$ if and only if there are vertices $\wt{m}_1,\dots,\wt{m}_i$ satisfying:
\begin{itemize}
\item $v=\wt{m}_i$
\item $\wt{m}_j$ is in the closed neighbourhood of $\wt{m}_{j-1}$ for every $2\leq j\leq i$
\item for every $2\leq j\leq i$, $d(c_j,\wt{m}_j)\leq d(c_{j-1},\wt{m}_{j-1})$ if and only if $b_j=1$.
\end{itemize}

Let $f:\bigcup_{i\in\bN}\{0,1\}^i\rightarrow V(G)$. We say that the cat \textit{follows strategy $(c_1,c_2,f)$} if $c_1$ and $c_2$ are the first two vertices it tests and $c_{i}=f(b_2,\dots,b_{i-1})$ for $i\geq3$. We say that \textit{the cat can localise the mouse up to distance $d$ within time $t$ on $G$} if there is some strategy $f$ such that, whenever the cat follows strategy $f$, there exists $i\leq t$ with $\rad_G(M_i)\leq d$, where the \textit{radius} of a set $W$ of vertices  is defined as
\begin{equation*}
\rad_G(W)=\min\big\{\max\{d(v,w):w\in W\}:v\in V(G)\big\}.
\end{equation*}

The cat only knows $G$ and the values $b_i$ (and hence can calculate the sets $M_i$). The mouse has more information. As well as $G$, the mouse also knows about the strategy followed by the cat. Therefore, the mouse knows how the cat will choose its vertices and can adapt its own strategy accordingly.

Note that the cat must follow a strategy $f:\bigcup_{i\in\bN}\{0,1\}^i\rightarrow V(G)$ which has been decided before the game begins. We do not allow the cat to follow a random strategy (as we require the cat to be certain of localising the mouse within given time and distance).

\section{Lower Bound}
\label{sec:lower}
Before we prove Theorem \ref{thm:lower} for every $n\in\bN$, we will first show that it holds for some special cases. We will prove the following proposition.
\begin{prop}
\label{prop:lower}
Let $t$ be a multiple of 12 and take $n=t^2+1$. There exists a tree $T$ on $n$ vertices such that, regardless of the strategy employed by the cat, the mouse is able to avoid being localised up to distance $\frac{t}{12}$ on $T$ for an infinite period of time.
\end{prop}
\noindent
Throughout the rest of this section, fix $n=t^2+1$ where $t$ is divisible by $12$. We define $T$ to be the tree on $n$ vertices formed from the star $K_{1,t}$ by subdividing each edge exactly $t-1$ times. We will call the centre vertex $u$ and refer to the subdivided edges as branches. We do not regard $u$ to be on any branch. The structure of $T$ means (informally) that, at each turn, the cat is only able to check whether the mouse lies on a single branch of the graph.

In the proof of Theorem \ref{prop:lower}, we will provide a strategy for the mouse where $m_i$ is always on a different branch to $c_{i-1}$ and $c_i$ and hence the cat can never determine from its tests which branch the mouse is on. We will also consider a second set of possible positions $(w_i)$ for the mouse such that the cat cannot tell the difference between a mouse at $m_i$ and a mouse at $w_i$. For every $i$, we will ensure that $d(m_i,w_i)>\frac{t}6$ and hence $\rad_G(M_i)>\frac{t}{12}$. This is sufficient to ensure that cat cannot localise the mouse up to distance $\frac{t}{12}$ in finite time.

\begin{proof}[Proof of Theorem \ref{prop:lower}]
Suppose that the cat is trying to localise the mouse to up distance $\frac{t}{12}$ on $T$. We will provide a strategy for the mouse such that $\rad_G(M_i)>\frac{t}{12}$ for every $i\in\bN$. Hence, the cat is never able to localise the mouse on $T$.

To simplify the numbers, the first turn is at time $t=0$. At this time, the cat only knows the graph $T$. The mouse knows what $T$ looks like as well as the strategy employed by the cat. The mouse will sometimes give the cat additional information about its position or movements. We will see later that this does not impact the result.

The mouse employs the following strategy, making the moves given by $(m_i)$. In what follows, we will also consider a second set of possible moves $(w_i)$ for the mouse. At each turn, we will ensure that both $m_i,w_i\in M_i$ and $d(m_i,w_i)>\frac{t}6$. Hence $\rad_G(M_i)>\frac{t}{12}$ as required.

At several stages in the mouse's strategy, it needs to choose a branch which the cat will not check for a given period of time $\tau$ where $\tau<t$. As $T$ has $t$ branches and the cat can only check one per turn, such a branch certainly exists. In order to find such a branch, it is necessary to look ahead at the cat's strategy. Note that, as long as the cat does not play on the branch within the given time period, it does not matter which branch the mouse chooses because its movements, and consequently the cat's movements, will be the same.

Here is the strategy which the mouse will follow:
\begin{enumerate}
\item The mouse chooses a branch which the cat will not check for $\frac{2t}3$ turns and chooses $m_0$ to be the vertex at distance $\frac t4$ from $u$ on this branch. Similarly, let $w_0$ be a vertex at distance $\frac t4$ from $u$ on another branch that the cat will not check for $\frac{2t}3$ turns. The mouse tells the cat that it is at distance $\frac t4$ from $u$ so $M_0=\{v\in V(T):d(v,u)=\frac t4\}$. The mouse also tells the cat $\frac t4$ branches which it is not located on (none of which include $m_0$ or $w_0$).

\item Over the next $\frac t6$ turns, the cat will gradually lose information about the position of the mouse. For $1\leq i\leq\frac t6$, define $m_i$ and $w_i$ as follows:
\begin{itemize}
\item If $d(c_i,u)\leq d(c_{i-1},u)$, then $m_i$ is one vertex closer to the $u$ than $m_{i-1}$ and $w_i=w_{i-1}$.
\item If $d(c_i,u)>d(c_{i-1},u)$, then $m_i=m_{i-1}$ and $w_i$ is one vertex further away from $u$ than $w_{i-1}$.
\end{itemize}
Note that, in the first case, we have $b_i=1$ and, in the second case, we have $b_i=0$. In both cases, the cat receives the same value of $b_i$ for a mouse at $m_{i-1}$ and then $m_i$ and a mouse at $w_{i-1}$ followed by $w_i$ and so cannot tell the difference between these positions. At time $\frac t6$, we have $\frac t3\leq d(m_{t/6},w_{t/6})\leq\frac{2t}3$.

\item Let $d=d(m_{t/6},u)$ so $\frac{t}{12}\leq d\leq\frac t4$. The mouse tells the cat that it will now run towards $u$ for $d$ time steps and proceeds to do so. For $\frac t6<i\leq\frac t6+d$, we have $m_i$ and $w_i$ are both one vertex closer to $u$ than $m_{i-1}$ and $w_{i-1}$ respectively.

\item At time $\frac t6+d$, we have $m_{t/6+d}=u$ and $d(w_{t/6+d},u)=\frac t6$. The mouse now chooses a branch on which the cat will not play in the next $\frac{11t}{12}$ turns (and different from the branch containing $w_i$). This is the branch on which the $(m_i)$ will now lie. The $(w_i)$ will remain on their original branch.

\item The mouse now tells the cat it is running away from the centre for $\frac t4$ time steps and does so. For $\frac t6+d<i\leq\frac{5t}{12}+d$, both $m_i$ and $w_i$ are one vertex further away from $u$ than $m_{i-1}$ and $w_i$ respectively.

\item Let $j=\frac{5t}{12}+d$. At time $j$, we have $d(m_j,u)=\frac t4$. Let $B$ be a branch which the cat has not checked in the previous $\frac t4$ turns and will not check for the next $\frac{2t}3$ turns, and which is different to the branch the mouse is currently on (i.e. $m_j\notin B$). Redefine $w_j$ to be the vertex at distance $\frac t4$ from $u$ on $B$. The mouse then tells the cat it is at distance $\frac t4$ from $u$. The situation is now exactly the same as in step 1.
\end{enumerate}
\noindent
Stage 6 is identical to stage 1, both in the positions of the mouse and the information known by the cat. Thus, the mouse may repeat stages 2 through 6 to avoid being localised by the cat indefinitely.

Let us now justify the mouse's choice of branch in Step 1. Provided that the mouse chooses a branch that the cat will not play on in the next $\frac{2t}3$ rounds, the sequence $(b_i)$ and the branches the cat plays on do not depend on which branch the mouse is on. Therefore, as the mouse knows the cat's strategy, the mouse can look into the future, simulate how the cat will play and hence find a suitable branch.

At stage 4, we have $m_{t/6+d}=u$. The cat has no idea which branch the mouse will run down and so all of the branches it previously eliminated now need checking again. However, the cat cannot tell the difference between a mouse at $m_{t/6+d}$ and one at $w_{t/6+d}$ and so it does not know that the mouse is at $u$. As $d(w_{t/6+d},u)=\frac t6$ but the cat does not know which branch $w_{t/6+d}$ is on, we still have $\rad_G(M_{t/6+d})>\frac{t}{12}$.

Between times $\frac t6+d$ and $j$, the cat has eliminated at most $\frac t4$ branches for a mouse that reached the centre at time $\frac t6+d$. At time $j$, the cat (depending on its play for the previous $\frac t4$ turns) may realise that the mouse was indeed at the centre at time $\frac t4+d$, but this is no longer helpful as the mouse is now at distance $\frac t4$ from the centre along an unknown branch.
\\
\\
Now consider the situation where the cat is not given any information except for $T$ and the values of the $b_i$ for $i\geq2$. The mouse can employ exactly the same strategy as it used above but without telling the cat anything. Let $M'_i$ be the possible positions of the mouse at time $i$ and $M_i$ be the sets given by the above argument where the cat has extra information. This additional information enables the cat to narrow down the mouse's possible position more accurately and so, for every $i$, we find that $M_i\subseteq M'_i$

Therefore, as the cat cannot localise the mouse up to distance $\frac t4$ when it is told some information about its position and movements, it definitely cannot localise the mouse when it does not have access to this extra knowledge.
\end{proof}
\noindent
Theorem \ref{thm:lower} follows as a corollary of Proposition \ref{prop:lower}.
\begin{proof}[Proof of Theorem \ref{thm:lower}]
Suppose $n=t^2+r$ where $t$ is a multiple of 12 and $r\in[24t+144]$. Let $T$ be the tree described in Proposition \ref{prop:lower} - the tree obtained by subdividing each edge of the star $K_{1,t}$ exactly $t-1$ times. Create the tree $T'$ on $n$ vertices by adding a branch $B$ of $r-1$ vertices to $T$. The cat and mouse play on $T'$. The mouse only plays on the subgraph $T$ and is never located on $B$. It uses the same strategy as in Proposition \ref{prop:lower} to avoid being localised up to distance $\frac{t}{12}$ on $T'$.
\end{proof}

\section{Upper Bound}
\label{sec:upper}
In this section, we will prove Theorem \ref{thm:upper}. In fact we prove a slightly more precise result.
\begin{theorem}
\label{thm:upper2}
\begin{itemize}
\item[(i)] Let $c > 0$. Then the cat can localise the mouse up to distance $\bigl(\frac{8}{c}+c\bigr) \sqrt{n}$ by time $\frac{4}{c}\sqrt{n}$.
\item[(ii)] The cat can localise the mouse up to distance $\frac{9}{2} \sqrt{n}$ by time $n$.
\end{itemize}
\end{theorem}
\noindent
Theorem \ref{thm:upper} follows from (i) by setting $c=\sqrt{8}$.

The theorem will follow from two lemmas giving upper bounds on how well a cat can localise a mouse in a graph.

The first lemma depends on how easily we can cover the vertices of the graph $G$ with balls. Thus we get stronger bounds when $G$ is a ``fat'' graph with lots of clustering.

\begin{lemma}
\label{lem:fat}
Let $L,k$ be natural numbers. Let $G=(V,E)$ be a graph whose vertices may be covered by $L$ balls of radius $k$. Then the cat can localise the mouse up to distance $4L+k$ in $G$ by time $2L$.
\end{lemma}
\noindent
Assuming this lemma, we can prove Theorem \ref{thm:upper2}(i).

\begin{proof}[Proof of Theorem \ref{thm:upper2} (i)]
Let $c>0$ and let $G=(V,E)$ be a connected graph. Let $\cR$ be a maximal subset of $V$ with pairwise distance at least $c\sqrt{n}$. Since the $\frac{c\sqrt{n}-1}{2}$-balls around vertices in $\cR$ are disjoint and $G$ is connected, $|\cR| \le \frac{2}{c}\sqrt{n}$. Thus the vertices of $G$ are covered by $\frac{2}{c}\sqrt{n}$ balls of radius $c\sqrt{n}$. So by Lemma \ref{lem:fat} the cat can localise the mouse up to distance $\left(\frac8c+c\right)\sqrt{n}$ by time $\frac{4}{c}\sqrt{n}$.
\end{proof}
\noindent
The second lemma depends on a different property of the graph $G$. The lemma gives stronger bounds when $G$ is a ``thin'' graph. We define the \textit{diameter} of any connected graph $G$ to be $\diam(G)=\max\{d(u,v):u,v\in V(G)\}$.

\begin{lemma}
\label{lem:thin}
Let $K$ be a natural number. Let $G=(V,E)$ be a graph such that for all $v\in V$, there exists an $\ell=\ell(v)<K$ such that $|\{w\in V:d(v,w)=\ell\}| <\frac\ell4$. Let $D$ be the diameter of $G$. Then the cat can localise the mouse up to distance $\frac{3K}2$ in $G$ by time $D/2$.
\end{lemma}
\noindent
Assuming this lemma, we can prove Theorem \ref{thm:upper2} (ii).

\begin{proof}[Proof of Theorem \ref{thm:upper2} (ii)]
Let $G=(V,E)$ be a graph with diameter $D$. For any $v\in V$, by the pigeonhole principle, there exists an $\ell<3\sqrt{n}$ such that $|\{w\in V:d(v,w)=\ell\}|<\frac\ell4$. We may then appeal to Lemma \ref{lem:thin} with $K=3\sqrt{n}$ and $D \le n$.
\end{proof}
\noindent
Let $L,k$ be natural numbers. Let $G=(V,E)$ be a graph and suppose $\{u_1,\ldots,u_L\}\subseteq V$ is such that $\bigcup_{i\in[L]}B_k(u_i)=V$. Consider the strategy for the cat given by Algorithm \ref{fatalgorithm}.
\\
\\
\begin{algorithm}[H]
\SetKw{KwFn}{Initialization}
    \KwFn{Set $i=1$, $w_i=1$}\;
    \While{$i<L$}{
        Set $c_{2i-1}=u_{w_i}$, $c_{2i}=u_{i+1}$\;
        \eIf{$d(c_{2i},m_{2i})\le d(c_{2i-1},m_{2i-1})$}{
            increase $i$ and set $w_i=i$\;
        }{
            increase $i$ and set $w_i=w_{i-1}$\;
        }
    }
\caption{Locating the mouse in a ``fat'' graph.}
\label{fatalgorithm}
\end{algorithm}
\vspace{\baselineskip}
\noindent
Note that if the mouse doesn't move, then this algorithm determines which of the $u_i$ is closest to the mouse. It takes them two at a time and chooses the closer one. When we run the same algorithm with a moving mouse, then the mouse may add noise, but only at a bounded rate. Lemma \ref{lem:fat} follows immediately from the following claim.

\begin{claim}
\label{claim:fat}
Independent of $(m_i)_{i\in\bN}$, we have $d(m_{2L-1},u_{w_L})\le4L+k$.
\end{claim}

\begin{proof}
Let $i\in[L-1]$. Then $w_{i+1}$ is either $w_i$ or $i+1$. In the former case, since the mouse moves along an edge once per time step,
\begin{align*}
d(u_{w_{i+1}},m_{2i+1})&=d(u_{w_i},m_{2i+1})\\
&\le d(u_{w_i},m_{2i-1})+2.
\end{align*}
Otherwise $d(u_i,m_{2i})<d(c_{u_i},m_{2i-1})$ and $w_i = i$ in which case
\begin{align*}
d(u_{w_{i+1}},m_{2i+1})&=d(u_i,m_{2i+1})\\
&\le d(u_i,m_{2i})+1\\
&\le d(u_{w_i},m_{2i-1})+1.
\end{align*}
In either case, we have $d(u_{w_{i+1}},m_{2i+1})\le d(u_{w_i},m_{2i-1})+2$. By induction, we then have for all $i \le L$
\begin{equation*}
d(u_{w_k},m_{2L-1})\le d(u_{w_i},m_{2i-1})+2(L-i).
\end{equation*}

Now for each $i =2,\ldots,L$, we have $d(u_{w_i},m_{2i-1}) \le d(u_i, m_{2i-2}) + 1$. Therefore for all $i = 2,\ldots,L$
\begin{align*}
d(u_i,m_{2L-1})&\ge d(u_i,m_{2i-2})-2(L-i)-1\\
&\ge d(u_{w_i},m_{2i-1})-2(L-i)-2\\
&\ge d(u_{w_L},m_{2i-1})-4(L-i)-2\\
&\ge d(u_{w_L},m_{2i-1})-4L.
\end{align*}
Since $w_1=1$, the above bound also holds for $i=1$. Rearranging this, we get that $d(u_{w_L},m_{2i-1})\le4L+d(u_i,m_{2L-1})$ for all $i\le L$. But $(B_k(u_i))_{i\in[L]}$ forms a cover of $V$ and so $d(u_i,m_{2L-1})$ must be at most $k$ for one of the $i$ and so $d(u_{w_k},m_{2i-1})\le 4L + k$. 
\end{proof}
\noindent
Let $K$ be a natural number. Let $G=(V,E)$ be a graph such that, for all $v\in V$, there exists an $\ell=\ell(v)<K$ such that $|\{w\in V:d(v,w)=\ell\}|<\frac\ell4$. Let $D$ be the diameter of $G$ and for each vertex $v$ fix $\ell(v)$ such that $|\{w\in V:d(v,w)=\ell\}|<\frac\ell4$. Consider Algorithm \ref{thinalgorithm} which dictates a strategy for the cat given the movements of the mouse $(m_i)_{i\in\bN}$. Note that $T_i$ is not necessary for the strategy but we include it to simplify the analysis.
\\
\\
\begin{algorithm}[H]
\SetKw{KwFn}{Initialization}
    \KwFn{Set $i=0$, $j=1$, $T_1=0$ and pick $v_1 \in V$ arbitrarily}\;
    \While{$i<D$}{
        Set $U = \{w\in V:d(v_j,w) = \ell(v_j)\}$ \;
        Set $u_j = v_j$ \;
        \While{$U \neq \emptyset$}{
            Increase $i$ \;
            Pick $w \in U$ \;
        	    Set $c_{2i-1}=u_j$, $c_{2i}=w$\;
            \If{$d(c_{2i},m_{2i})\le d(c_{2i-1},m_{2i-1})$}{
                Set $u_j = w$\;
        	    }
            Set $U = U - w$\;
            }
        Set $T_{j+1} = i$, $v_{j+1} = u_j$\;
        Increase $j$     
    }
\caption{Locating the mouse in a ``thin'' graph.}
\label{thinalgorithm}
\end{algorithm}
\vspace{\baselineskip}
\noindent
Lemma \ref{lem:thin} follows immediately from the following claim.
\begin{claim}
\label{claim:thin}
$d(v_n,m_{2T_n-1}) \le \frac{3}{2}K$ for all $T_n \ge \frac{D}{2}-\frac{3}{4}K$.
\end{claim}

\begin{proof}
Suppose we run Algorithm \ref{thinalgorithm}. We claim that, for all $n\ge 1$, $T_{n+1}-T_n \le \frac{\ell(v)}{4}$ and
\begin{equation}
\label{ineq}
d(v_{n+1},m_{2T_{n+1}-1}) \le \max\left\{d(v_n,m_{2T_n-1})-\frac{\ell(v_n)}{2}, \frac{3}{2}K\right\}.
\end{equation}
Assuming these inequalities, we see that $T_n \le \frac{1}{4}\sum_{i=1}^{n-1}\ell(v_i)$ and
\begin{equation*}
d(v_n,m_{2T_n-1}) \le \max\left\{d(v_1,m_1) - \frac{1}{2}\sum_{i=1}^{n-1}\ell(v_i), \frac32K\right\}.
\end{equation*}
Since we have $d(v_1,m_1) \le D$, it follows that, if $\frac{1}{2}\sum_{i=1}^{n-1}\ell(v_i) \ge D-\frac32K$, then $d(v_n,m_{2T_n-1}) \le\frac32K$. Therefore, if $T_n \ge\frac12D-\frac34K$, then $d(v_n,m_{2T_n-1}) \le\frac32K$.

It remains to show that for all $n \ge 1$, $T_{n+1}-T_n \le \frac{\ell(v)}{4}$ and \eqref{ineq} holds. First suppose that $d(v_n,m_{2T_n-1}) \ge K$. Let $\ell<K$ be such that $|\{w\in V:d(v_n,w)=\ell\}|<\frac\ell4$. Note that there must be some $w$ such that $d(v_n,w)=\ell$ and $w$ is in a shortest $v_n-m_{2T_n-1}$ path. So then $d(w,m_{2T_n-1})=d(v_j,m_{2T_n-1})-\ell$. Let $i'\ge i$ be such that $c_{2i}=w$. Then
\begin{align*}
d(c_{2i},m_{2i})&\le d(w,m_{2T_n-1})+2i-(2T_n-1)\\
&=d(v_n,m_{2T_n-1})-\ell+2i-(2T_n-1).
\end{align*}
On the other hand, $d(v_{n+1},m_{2T_{n+1}-1})\le d(c_{2i},m_{2i})+2T_{n+1}-1-2i$, and so
\begin{align*}
d(v_{n+1},m_{2T_{n+1}-1})&\le d(v_n,m_{2T_n-1})-\ell+2i-(2T_n-1)+2T_{n+1}-1-2i\\
&=d(v_n,m_{2T_n-1})+2(T_{n+1}-T_n)-\ell.
\end{align*}
$T_{n+1}-T_n=|\{w\in V:d(v_n,w)=\ell\}|$ and so we have that
\begin{equation*}
d(v_{n+1},m_{2T_{n+1}-1})\le d(v_n,m_{2T_n-1})-\frac\ell2.
\end{equation*}

Now suppose that $d(v_n,m_{2T_n-1})<K$. Since $T_{n+1}-T_n\le\frac K4$, the mouse moves at most $\frac K2$ steps. Therefore 
\begin{equation*}
d(v_{n+1},m_{2T_{n+1}-1})\le d(v_n,m_{2T_n-1})+\frac K2\le\frac{3K}2.
\end{equation*}

In either case, a simple induction on $n \ge 1$, we have $T_{n+1}-T_n \le\frac14\ell(v_{T_n})$ and that either $d(v_{T_n},m_{2T_n-1}) \le\frac{3K}2$ or
\begin{equation*}
d(v_{n+1},m_{2T_{n+1}-1})\le d(v_n,m_{2T_n-1})-\frac\ell2.
\end{equation*}
\end{proof}

\section{Conclusion}
\label{sec:conclusion}
While we have established $\Theta(\sqrt{n})$ as a general upper bound for a cat localising a mouse on a graph, it would be nice to get better bounds depending on the structure of $G$. Trivially we have $\diam(G)$ as an upper bound. In addition to this, when the diameter of $G$ is $n - R$, we may appeal to Lemma \ref{lem:thin} with $K = O(R)$ to get that the cat can localise the mouse up to distance $O(R)$. Unfortunately, nothing more can be said when the diameter is between these two extreme cases. Indeed, consider lengthening two of the branches of the tree given in Section \ref{sec:lower}. Lemmas \ref{lem:fat} and \ref{lem:thin} are perhaps steps in the right direction here. Lemma \ref{lem:fat} works well when there is lots of clustering whilst Lemma \ref{lem:thin} works well when there is very little clustering.

Another interesting direction to consider is the case where multiple cats are trying to localise the mouse exactly. In particular, suppose that there is a collection\footnote{It seems there is a large choice of collective nouns for a collection of cats, including clowder, clutter, destruction (wild cats only), dout, glare, glorying, nuisance, and pounce \cite{T17}.} of $k$ cats $c^1,\dots,c^k$. At turn $i$, each cat $c^j$ chooses a vertex $c^j_i$ and is told the value of $b_i^j$. There is a simple argument which shows that 7 cats is sufficient to localise the mouse exactly on any tree in finite time. We believe that this is not optimal and suggest the following questions.
\begin{question}
What is the minimum number of cats that are needed to localise the mouse exactly on any tree?
\end{question}
\begin{question}
What happens on a general graph $G$? How accurately can $k$ cats localise a mouse?
\end{question}

\end{document}